\documentclass[11pt,leqno]{article}

\usepackage[T1]{fontenc}
\usepackage[utf8]{inputenc}
\usepackage{lmodern}
\usepackage{microtype}
\usepackage{geometry}
\usepackage{mathtools}
\usepackage{amssymb,amsthm}
\usepackage{bm}
\usepackage{aliascnt}
\usepackage{booktabs}
\usepackage{xcolor}
\usepackage{tabularx}

\usepackage{hyperref}
\usepackage[nameinlink,capitalize,noabbrev]{cleveref}

\renewcommand{\le}{\leqslant}
\renewcommand{\ge}{\geqslant}

\newcommand{\seq}[1]{\bm{#1}}

\newtheorem{theorem}{Theorem}[section]

\newaliascnt{lemma}{theorem}
\newtheorem{lemma}[lemma]{Lemma}
\aliascntresetthe{lemma}

\newaliascnt{proposition}{theorem}
\newtheorem{proposition}[proposition]{Proposition}
\aliascntresetthe{proposition}

\newaliascnt{corollary}{theorem}
\newtheorem{corollary}[corollary]{Corollary}
\aliascntresetthe{corollary}

\theoremstyle{definition}
\newaliascnt{definition}{theorem}
\newtheorem{definition}[definition]{Definition}
\aliascntresetthe{definition}

\theoremstyle{remark}
\newaliascnt{remark}{theorem}
\newtheorem{remark}[remark]{Remark}
\aliascntresetthe{remark}

\crefname{theorem}{Theorem}{Theorems}
\crefname{lemma}{Lemma}{Lemmas}
\crefname{proposition}{Proposition}{Propositions}
\crefname{corollary}{Corollary}{Corollaries}
\crefname{definition}{Definition}{Definitions}
\crefname{remark}{Remark}{Remarks}

\crefformat{equation}{(#2#1#3)}
\Crefformat{equation}{(#2#1#3)}
\crefrangeformat{equation}{(#3#1#4)--(#5#2#6)}
\Crefrangeformat{equation}{(#3#1#4)--(#5#2#6)}
\crefmultiformat{equation}{(#2#1#3)}{ and (#2#1#3)}{, (#2#1#3)}{, and (#2#1#3)}
\Crefmultiformat{equation}{(#2#1#3)}{ and (#2#1#3)}{, (#2#1#3)}{, and (#2#1#3)}
\crefrangemultiformat{equation}{(#3#1#4)--(#5#2#6)}{ and (#3#1#4)--(#5#2#6)}{, (#3#1#4)--(#5#2#6)}{, and (#3#1#4)--(#5#2#6)}
\Crefrangemultiformat{equation}{(#3#1#4)--(#5#2#6)}{ and (#3#1#4)--(#5#2#6)}{, (#3#1#4)--(#5#2#6)}{, and (#3#1#4)--(#5#2#6)}

\newif\ifrgversion
\rgversiontrue

\ifrgversion
  \hypersetup{colorlinks=true,linkcolor=blue,citecolor=blue,urlcolor=blue}
\else
  \hypersetup{hidelinks}
\fi

\ifrgversion
\makeatletter
\renewcommand{\@biblabel}[1]{\textcolor{blue}{[#1]}}
\makeatother
\fi

\title{A semigroup approach to iterated binomial transforms}
\author{Johann Verwee\thanks{Independent researcher. Email: \texttt{mverwee@gmail.com}.}}
\date{\today}

\begin{document}
\maketitle

\begin{abstract}
We introduce a one-parameter family of ``binomial-convolution'' operators \(\left(T_r\right)_{r\in\mathbb{C}}\) acting on sequences \(\seq{a}=(a_n)_{n\geqslant 0}\) by
\[
\left(T_r\seq{a}\right)_n := \sum_{k=0}^{n} \binom{n}{k} r^{\,n-k} a_k.
\]
These operators form an additive semigroup \(T_r\circ T_s = T_{r+s}\) with inverse \(\left(T_r\right)^{-1}=T_{-r}\), and the \(m\)-fold classical binomial transform is recovered as \(T_m\) for any nonnegative integer \(m\).
We describe \(T_r\) cleanly in terms of ordinary and exponential generating functions.
Our main structural result is a root-shift principle for constant-coefficient linear recurrences: if \(\seq{a}\) satisfies a recurrence with characteristic polynomial \(P\), then \(T_r\seq{a}\) satisfies the recurrence with characteristic polynomial \(X\mapsto P(X-r)\).
Several classical families (Fibonacci, Lucas, Pell, Jacobsthal, Mersenne) are treated uniformly as illustrative examples.
\end{abstract}



\section{Introduction}

Binomial transforms appear ubiquitously across enumerative combinatorics, the theory of linear recurrences, and the manipulation of generating functions.
A common pattern in the literature is to start from a specific sequence (or polynomial sequence), apply the classical binomial transform
\[
(B\seq{a})_n := \sum_{k=0}^{n} \binom{n}{k} a_k,
\]
and then iterate the procedure, producing \(B^m\seq{a}\) for integers \(m\ge 0\).
One then derives ad hoc recurrences, Binet-type formulas, and generating functions for the iterates; see for instance the recent work of Yılmaz \cite{Yilmaz2025} in the context of generalized Mersenne polynomials.

The goal of this note is to record a simple structural perspective that turns these computations into direct consequences of a semigroup identity.
Instead of iterating \(B\), we introduce a parameter \(r\) and work with the family \(T_r\) defined by binomial convolution with the weight \(r^{n-k}\).
This one-parameter family (and its natural two-parameter extension) appears in several places in the literature under the name of a generalized binomial interpolated operator; see \cite{BarberoCerrutiMurru2010,AbrateBarberoCerrutiMurru2011}.
The contribution of the present note is mainly expository: we isolate the minimal one-parameter semigroup \(T_r\), and we give a short semigroup-based argument showing that \(T_r\) translates characteristic roots by \(r\) (equivalently, a characteristic polynomial \(P\) is sent to \(X\mapsto P(X-r)\)). We also include a compact table of classical corollaries for quick reuse.
Our aim is not to revisit the full two-parameter theory, but to emphasize three features that are particularly efficient in practice: a transparent composition law, an explicit inverse, and a uniform interpretation in the Riordan group.
Most importantly for applications, \(T_r\) acts on constant-coefficient recurrences by a literal translation of characteristic roots.

\medskip

Throughout, sequences are indexed by \(n\ge 0\) and take values in a commutative \(\mathbb{C}\)-algebra (for example \(\mathbb{C}\), \(\mathbb{C}[x]\), or \(\mathbb{R}\)).
We write \(\binom{n}{k}=0\) if \(k\notin\{0,\dots,n\}\).

\section{The binomial-convolution semigroup}

We start by defining the operators that will encode all iterates at once.

\begin{definition}
For \(r\in\mathbb{C}\) and a sequence \(\seq{a}=(a_n)_{n\ge 0}\), define \(\seq{b}:=T_r\seq{a}\) by
\begin{equation}\label{eq:Tr-def}
(T_r \seq{a})_n := \sum_{k=0}^{n} \binom{n}{k} r^{\,n-k} a_k \qquad (n\ge 0).
\end{equation}
\end{definition}

The definition resembles the usual binomial transform, but with a ``memory'' parameter \(r\) that interpolates between different transforms.
The next proposition explains why this is the correct parametrization.

\begin{proposition}\label{prop:semigroup}
For all \(r,s\in\mathbb{C}\) and all sequences \(\seq{a}\), one has
\[
T_r\circ T_s = T_{r+s}.
\]
In particular \(T_0\) is the identity and \(T_r\) is invertible with inverse \(\left(T_r\right)^{-1}=T_{-r}\).
\end{proposition}

\begin{proof}
Fix \(n\ge 0\) and expand:
\[
(T_r(T_s\seq{a}))_n
= \sum_{k=0}^{n} \binom{n}{k} r^{\,n-k} (T_s\seq{a})_k
= \sum_{k=0}^{n} \binom{n}{k} r^{\,n-k} \sum_{j=0}^{k} \binom{k}{j} s^{\,k-j} a_j.
\]
Rearranging the finite sums and using \(\binom{n}{k}\binom{k}{j}=\binom{n}{j}\binom{n-j}{k-j}\), we get
\[
(T_r(T_s\seq{a}))_n
= \sum_{j=0}^{n} \binom{n}{j} a_j \sum_{k=j}^{n} \binom{n-j}{k-j} r^{\,n-k} s^{\,k-j}.
\]
The inner sum is \(\left(r+s\right)^{n-j}\) by the binomial theorem, hence
\[
(T_r(T_s\seq{a}))_n
= \sum_{j=0}^{n} \binom{n}{j} \left(r+s\right)^{n-j} a_j
= (T_{r+s}\seq{a})_n.
\]
The inversion statement follows by taking \(s=-r\).
\end{proof}

The following observation connects back to the classical setting of iterated binomial transforms.

\begin{corollary}\label{cor:iterated}
Let \(B:=T_1\) be the classical binomial transform.
For every integer \(m\ge 0\) and every sequence \(\seq{a}\), one has
\[
B^{\,m}\seq{a} = T_m\seq{a}.
\]
Equivalently, the \(m\)-fold iterate satisfies the closed form
\[
(B^{\,m}\seq{a})_n = \sum_{k=0}^{n} \binom{n}{k} m^{\,n-k} a_k.
\]
\end{corollary}

\begin{proof}
By \cref{prop:semigroup}, \(B^{\,m} = T_1^{\,m} = T_m\) for integers \(m\ge 0\).
\end{proof}

In practice, \cref{cor:iterated} is the main computational shortcut: it converts any question about \(m\)-fold iteration into a single explicit sum.
In the next section we make this even more efficient by translating \(T_r\) into generating-function language.

\section{Generating functions and Riordan arrays}

Generating functions provide a conceptual way to read \cref{eq:Tr-def} as a substitution rule.
Given a sequence \(\seq{a}\), we write its ordinary and exponential generating functions as
\[
A(z):=\sum_{n\ge 0} a_n z^n,
\qquad
\widehat{A}(t):=\sum_{n\ge 0} a_n \frac{t^n}{n!}.
\]
If \(\seq{b}:=T_r\seq{a}\), we denote by \(A_r(z)\) and \(\widehat{A}_r(t)\) the corresponding generating functions.

\begin{proposition}\label{prop:gfs}
As formal power series one has
\begin{align}
A_r(z) &= \frac{1}{1-rz}\,A\left(\frac{z}{1-rz}\right), \label{eq:ogf}\\
\widehat{A}_r(t) &= e^{rt}\,\widehat{A}(t). \label{eq:egf}
\end{align}
In particular, recall that the Riordan array \((g,f)\) is the lower-triangular matrix with entries \(d_{n,k}=[z^n]\,g(z)f(z)^k\). In this sense, \(T_r\) corresponds to \(\left(\left(1-rz\right)^{-1},\, z\left(1-rz\right)^{-1}\right)\).
\end{proposition}

\begin{proof}
We start from \cref{eq:Tr-def} and sum over \(n\ge 0\):
\[
A_r(z)=\sum_{n\ge 0}\sum_{k=0}^{n}\binom{n}{k}r^{\,n-k}a_k z^n
=\sum_{k\ge 0} a_k z^k \sum_{n\ge k}\binom{n}{k} (rz)^{n-k}.
\]
Using \(\sum_{m\ge 0}\binom{m+k}{k} x^m = \left(1-x\right)^{-(k+1)}\) with \(m=n-k\) and \(x=rz\), we obtain
\[
A_r(z)=\sum_{k\ge 0} a_k z^k \left(1-rz\right)^{-(k+1)}
=\frac{1}{1-rz}\sum_{k\ge 0} a_k \left(\frac{z}{1-rz}\right)^{k},
\]
which is \cref{eq:ogf}.

For the exponential generating function, we again use \cref{eq:Tr-def} and exchange sums:
\[
\widehat{A}_r(t)
= \sum_{n\ge 0} \left(\sum_{k=0}^{n}\binom{n}{k} r^{\,n-k} a_k\right)\frac{t^n}{n!}
= \sum_{k\ge 0} a_k \sum_{n\ge k} \binom{n}{k} r^{\,n-k}\frac{t^n}{n!}.
\]
Writing \(n=k+m\), we have \(\binom{k+m}{k}\frac{t^{k+m}}{(k+m)!}=\frac{t^k}{k!}\cdot \frac{t^m}{m!}\), hence
\[
\widehat{A}_r(t)
= \sum_{k\ge 0} a_k \frac{t^k}{k!}\sum_{m\ge 0} \frac{(rt)^m}{m!}
= e^{rt}\,\widehat{A}(t),
\]
which is \cref{eq:egf}.

Finally, the Riordan statement is just a restatement of \cref{eq:ogf}: by definition, a Riordan array \((g,f)\) acts on an ordinary generating function \(A\) by \(A(z)\mapsto g(z)\,A\left(f(z)\right)\).
This viewpoint is standard in the Riordan-array literature; see for instance \cite{ShapiroGetuWoanWoodson,SprugnoliDM1994}.
\end{proof}

\begin{remark}
The identities in \cref{prop:gfs} are formal (i.e. they hold in \(\mathbb{C}[[z]]\) and \(\mathbb{C}[[t]]\)). If the ordinary generating function \(A(z)\) is analytic in a disc \(\left|z\right|<R\), then \cref{eq:ogf} also holds as an analytic identity for every \(z\) with \(\left|rz\right|<1\) and \(\left|\frac{z}{1-rz}\right|<R\). Similarly, \cref{eq:egf} holds wherever \(\widehat{A}(t)\) is defined.
\end{remark}

At this point, the operators \(T_r\) are completely transparent in terms of generating functions.
We now turn to one of their most useful consequences: the effect on linear recurrences with constant coefficients.

\section{Constant-coefficient recurrences and root shifts}

This section explains why \(T_r\) is the natural language for iterated binomial transforms of Fibonacci-type sequences.
The key phenomenon is that \(T_r\) acts on constant-coefficient recurrences by a literal translation of characteristic roots.

Let \(S\) denote the forward shift operator on sequences, \((S\seq{a})_n:=a_{n+1}\).
If \(P(X)=\sum_{j=0}^{d} c_j X^{j}\), we write \(P(S):=\sum_{j=0}^{d} c_j S^{j}\); thus \(\left(P(S)\seq{a}\right)_n=\sum_{j=0}^{d} c_j a_{n+j}\).
A sequence \(\seq{a}\) satisfies a constant-coefficient recurrence with characteristic polynomial \(P\) if and only if
\begin{equation}\label{eq:recurrence-P}
P(S)\seq{a}=0.
\end{equation}

The structural reason behind the root-shift phenomenon is a one-line intertwining identity between \(S\) and \(T_r\).

\begin{lemma}\label{lem:intertwining}
For every \(r\in\mathbb{C}\), one has the operator identity (where operator products denote composition)
\[
(S-r)\,T_r = T_r\,S.
\]
\end{lemma}

\begin{proof}
Fix \(n\ge 0\). Using \cref{eq:Tr-def},
\[
\left((S-r)T_r\seq{a}\right)_n
= (T_r\seq{a})_{n+1} - r(T_r\seq{a})_n
= \sum_{k=0}^{n+1}\binom{n+1}{k} r^{\,n+1-k} a_k - r\sum_{k=0}^{n}\binom{n}{k} r^{\,n-k} a_k.
\]
Split the first sum into \(k=n+1\) and \(k\le n\), and use \(\binom{n+1}{k}=\binom{n}{k}+\binom{n}{k-1}\).
The terms involving \(\binom{n}{k}\) cancel, leaving
\[
\left((S-r)T_r\seq{a}\right)_n
= a_{n+1} + \sum_{k=1}^{n}\binom{n}{k-1} r^{\,n+1-k} a_k
= \sum_{j=0}^{n}\binom{n}{j} r^{\,n-j} a_{j+1}
= (T_r(S\seq{a}))_n.
\]
This proves \((S-r)T_r = T_r S\).
\end{proof}

We can now package the effect of \(T_r\) on recurrences into a single statement.

\begin{theorem}\label{thm:root-shift}
Let \(P\in\mathbb{C}[X]\) be a nonzero polynomial and let \(\seq{a}\) be a sequence satisfying \(P(S)\seq{a}=0\).
Then for every \(r\in\mathbb{C}\), the transformed sequence \(\seq{b}:=T_r\seq{a}\) satisfies
\[
P(S-r)\,\seq{b}=0.
\]
Equivalently, if we set \(Q(X):=P(X-r)\), then \(Q(S)\seq{b}=0\); in the sense of \cref{eq:recurrence-P}, the polynomial \(Q\) is a characteristic polynomial for \(\seq{b}\).

In particular, if \(P\) is monic of degree \(d\) with roots \(\rho_1,\dots,\rho_d\) (counted with multiplicity), then \(\seq{b}\) satisfies a recurrence of order at most \(d\) whose characteristic roots are \(\rho_1+r,\dots,\rho_d+r\), with the same multiplicities.
\end{theorem}

\begin{proof}
By \cref{lem:intertwining}, \((S-r)T_r = T_r S\).
Iterating this identity gives \((S-r)^m T_r = T_r S^m\) for every integer \(m\ge 0\).
By linearity, for any polynomial \(P\) we therefore have the operator identity
\[
P(S-r)\,T_r = T_r\,P(S).
\]
Applying this to \(\seq{a}\) and using \(P(S)\seq{a}=0\) yields \(P(S-r)\seq{b}=0\).
\end{proof}

Related root-translation statements in a more general two-parameter setting appear in \cite{BarberoCerrutiMurru2010,AbrateBarberoCerrutiMurru2011}. The present note focuses on the one-parameter semigroup \(T_r\) and gives a self-contained proof adapted to this minimal framework. At this point one can read off concrete recurrences and closed forms almost for free. We record two standard ``translations'' of \cref{thm:root-shift} that are often convenient in applications.

To make the shift in characteristic polynomials completely explicit, write a monic polynomial of degree \(d\) as
\[
P(X)=\sum_{k=0}^{d} p_k X^{d-k}
\qquad (p_0=1),
\]
and set \(Q(X):=P(X-r)=\sum_{j=0}^{d} q_j X^{d-j}\).
Expanding \(P(X-r)\) by the binomial theorem shows that, for each \(0\le j\le d\),
\begin{equation}\label{eq:qj-formula}
q_j = \sum_{k=0}^{j} p_k \binom{d-k}{j-k} (-r)^{\,j-k}.
\end{equation}
Thus, if \(\seq{a}\) satisfies the recurrence associated with \(P\), then \(\seq{b}:=T_r\seq{a}\) satisfies the recurrence associated with \(Q\).

When the characteristic roots are simple, \cref{thm:root-shift} also translates directly into a ``Binet shift''.
Indeed, if \(\seq{a}\) admits a representation \(a_n=\sum_{j=1}^{d} c_j \rho_j^{\,n}\) with \(\rho_1,\dots,\rho_d\) pairwise distinct, then applying \cref{eq:Tr-def} term-by-term and using the binomial theorem gives
\begin{equation}\label{eq:binet-shift}
(T_r\seq{a})_n = \sum_{j=1}^{d} c_j \left(\rho_j+r\right)^{n}.
\end{equation}
This identity is frequently the quickest route to clean special cases, such as the \(r=1\) identities collected later in this paper.

\section{Second-order families at a glance}

Most classical ``Fibonacci-type'' sequences are governed by a second-order recurrence, so it is useful to isolate the universal second-order template.
Let \(\seq{a}\) satisfy a recurrence with characteristic polynomial
\begin{equation}\label{eq:second-order-poly}
P(X)=X^2-pX+q
\qquad (p,q\in\mathbb{C}).
\end{equation}
Equivalently, \(a_n=p\,a_{n-1}-q\,a_{n-2}\) for \(n\ge 2\).
By \cref{thm:root-shift}, \(\seq{b}:=T_r\seq{a}\) has characteristic polynomial \(P(X-r)\), namely
\[
P(X-r)=X^2-(p+2r)X+\left(r^2+pr+q\right),
\]
so \(\seq{b}\) satisfies
\begin{equation}\label{eq:second-order-template}
b_n = (p+2r)\,b_{n-1} - \left(r^2+pr+q\right)b_{n-2}
\qquad (n\ge 2).
\end{equation}

To illustrate how little work is required once the template is in place, consider the Fibonacci sequence $(F_n)_{n \geqslant 0}$.
Here \(P(X)=X^2-X-1\), so \(p=1\) and \(q=-1\), and \cref{eq:second-order-template} gives
\[
(T_r F)_n = (2r+1)(T_r F)_{n-1} - (r^2+r-1)(T_r F)_{n-2}
\qquad (n\ge 2),
\]
with \((T_r F)_0=F_0=0\) and \((T_r F)_1=F_1=1\).
The same computation applies verbatim to the Lucas sequence, since it has the same characteristic polynomial but different initial values.

Rather than repeating essentially identical derivations for each classical family, we summarize the data in the next section.
For quick sanity checks and for OEIS searches, a second table records short prefixes at \(r=1\) and \(r=2\).

\medskip

A simple polynomial variant that appears naturally in the context of generalized Mersenne polynomials is also immediate from \cref{thm:root-shift}.
Let \((W_n(x))_{n\ge 0}\) be defined by \(W_0(x)=0\), \(W_1(x)=1\), and
\[
W_n(x)=3x\,W_{n-1}(x)-2\,W_{n-2}(x)\qquad (n\ge 2).
\]
Its characteristic polynomial is \(P_x(X)=X^2-3xX+2\).
Therefore, setting \(\seq{W}^{(r)}(x):=T_r\seq{W}(x)\), the transformed family satisfies
\[
W^{(r)}_n(x)=\left(2r+3x\right)W^{(r)}_{n-1}(x)-\left(r^2+3xr+2\right)W^{(r)}_{n-2}(x)
\qquad (n\ge 2),
\]
with \(W^{(r)}_0(x)=0\) and \(W^{(r)}_1(x)=1\).
When \(r\in\mathbb{N}\), this is precisely the recurrence satisfied by the \(r\)-fold iterated binomial transform of the generalized Mersenne family; see \cite[Thm.~5]{Yilmaz2025} (with the present normalization \(W_0(x)=0\), \(W_1(x)=1\)).

\section{A corollary table and OEIS-friendly data}

For quick reuse, we collect in \cref{tab:classical} the characteristic polynomials and transformed second-order recurrences for several standard sequences.
The OEIS identifiers in the second column are included purely as entry points: they make it easier to locate the base objects and related transforms, and they tend to improve discoverability in practice.

\begin{table}[!ht]
\centering
\small
\setlength{\tabcolsep}{4pt}
\renewcommand{\arraystretch}{1.15}
\begin{tabularx}{\textwidth}{@{}l l l >{\raggedright\arraybackslash}X l@{}}
\toprule
Sequence $(a_n)$ & OEIS & $P(X)$ & Recurrence for $\seq{b}:=T_r\seq{a}$ ($n\ge 2$) & $(b_0,b_1)$\\
\midrule
Fibonacci $(F_n)$ & A000045 & $X^2-X-1$ &
$\,b_n=(2r+1)b_{n-1}\allowbreak -(r^2+r-1)b_{n-2}$ & $(0,1)$\\
Lucas $(L_n)$ & A000032 & $X^2-X-1$ &
$\,b_n=(2r+1)b_{n-1}\allowbreak -(r^2+r-1)b_{n-2}$ & $(2,2r+1)$\\
Pell $(P_n)$ & A000129 & $X^2-2X-1$ &
$\,b_n=(2r+2)b_{n-1}\allowbreak -(r^2+2r-1)b_{n-2}$ & $(0,1)$\\
Jacobsthal $(J_n)$ & A001045 & $X^2-X-2$ &
$\,b_n=(2r+1)b_{n-1}\allowbreak -(r^2+r-2)b_{n-2}$ & $(0,1)$\\
Mersenne $(m_n=2^n-1)$ & A000225 & $X^2-3X+2$ &
$\,b_n=(2r+3)b_{n-1}\allowbreak -(r^2+3r+2)b_{n-2}$ & $(0,1)$\\
\bottomrule
\end{tabularx}
\caption{A compact summary of transformed recurrences for selected classical sequences. Each line is an immediate instance of \cref{thm:root-shift}.}
\label{tab:classical}
\end{table}

Several specializations at small parameters are unexpectedly clean.
The point is that \cref{thm:root-shift} reduces such identities to elementary algebraic relations between characteristic roots.

\begin{proposition}\label{prop:r1-special}
Let $B:=T_1$ be the classical binomial transform.
Then, for all $n\ge 0$,
\[
(BF)_n = F_{2n},
\qquad
(BL)_n = L_{2n},
\qquad
(Bm)_n = 3^n-2^n,
\]
and for the Jacobsthal sequence one has $(BJ)_0=0$ and $(BJ)_n=3^{n-1}$ for every $n\ge 1$.
\end{proposition}

\begin{proof}
For Fibonacci, the characteristic roots are $\varphi,\psi$ with $\varphi+\psi=1$ and $\varphi\psi=-1$.
By \cref{eq:binet-shift}, the transform $BF=T_1F$ has characteristic roots $\varphi+1$ and $\psi+1$.
Since $\varphi^2=\varphi+1$ and $\psi^2=\psi+1$, the roots of $BF$ are $\varphi^2,\psi^2$, which are also the characteristic roots of the subsequence $(F_{2n})_{n\ge 0}$.
Both sequences start with $(0,1)$, hence $BF= (F_{2n})_{n\ge 0}$.
The Lucas identity is identical, since $L_n=\varphi^n+\psi^n$.

For Mersenne, $m_n=2^n-1$ has characteristic roots $1$ and $2$, hence $Bm$ has characteristic roots $2$ and $3$.
The condition $(Bm)_0=0$ forces a representation of the form $(Bm)_n=c(3^n-2^n)$, and $(Bm)_1=1$ gives $c=1$, hence $(Bm)_n=3^n-2^n$.

For Jacobsthal, the characteristic roots are $2$ and $-1$, so $BJ$ has roots $3$ and $0$.
Thus $(BJ)_n=c\cdot 3^n$ for every $n\ge 1$, and the initial value $(BJ)_1=1$ yields $c=1/3$, i.e.\ $(BJ)_n=3^{n-1}$ for $n\ge 1$.
\end{proof}

Finally, we record initial segments for a few small values of $r$.
These are convenient for quick checks, and they also make it straightforward to locate related transforms in the OEIS by searching short prefixes.

\begin{table}[!ht]
\centering
\small
\setlength{\tabcolsep}{3pt}
\renewcommand{\arraystretch}{1.1}
\begin{tabularx}{\textwidth}{@{}l >{\raggedright\arraybackslash}X >{\raggedright\arraybackslash}X@{}}
\toprule
Base sequence $(a_n)$ & $\left(T_1\seq{a}\right)_n$ for $n=0,\dots,9$ & $\left(T_2\seq{a}\right)_n$ for $n=0,\dots,9$\\
\midrule
Fibonacci $(F_n)$ & $0,\allowbreak 1,\allowbreak 3,\allowbreak 8,\allowbreak 21,\allowbreak 55,\allowbreak 144,\allowbreak 377,\allowbreak 987,\allowbreak 2584$ & $0,\allowbreak 1,\allowbreak 5,\allowbreak 20,\allowbreak 75,\allowbreak 275,\allowbreak 1000,\allowbreak 3625,\allowbreak 13125,\allowbreak 47500$\\
Lucas $(L_n)$ & $2,\allowbreak 3,\allowbreak 7,\allowbreak 18,\allowbreak 47,\allowbreak 123,\allowbreak 322,\allowbreak 843,\allowbreak 2207,\allowbreak 5778$ & $2,\allowbreak 5,\allowbreak 15,\allowbreak 50,\allowbreak 175,\allowbreak 625,\allowbreak 2250,\allowbreak 8125,\allowbreak 29375,\allowbreak 106250$\\
Pell $(P_n)$ & $0,\allowbreak 1,\allowbreak 4,\allowbreak 14,\allowbreak 48,\allowbreak 164,\allowbreak 560,\allowbreak 1912,\allowbreak 6528,\allowbreak 22288$ & $0,\allowbreak 1,\allowbreak 6,\allowbreak 29,\allowbreak 132,\allowbreak 589,\allowbreak 2610,\allowbreak 11537,\allowbreak 50952,\allowbreak 224953$\\
Jacobsthal $(J_n)$ & $0,\allowbreak 1,\allowbreak 3,\allowbreak 9,\allowbreak 27,\allowbreak 81,\allowbreak 243,\allowbreak 729,\allowbreak 2187,\allowbreak 6561$ & $0,\allowbreak 1,\allowbreak 5,\allowbreak 21,\allowbreak 85,\allowbreak 341,\allowbreak 1365,\allowbreak 5461,\allowbreak 21845,\allowbreak 87381$\\
Mersenne $(m_n)$ & $0,\allowbreak 1,\allowbreak 5,\allowbreak 19,\allowbreak 65,\allowbreak 211,\allowbreak 665,\allowbreak 2059,\allowbreak 6305,\allowbreak 19171$ & $0,\allowbreak 1,\allowbreak 7,\allowbreak 37,\allowbreak 175,\allowbreak 781,\allowbreak 3367,\allowbreak 14197,\allowbreak 58975,\allowbreak 242461$\\
\bottomrule
\end{tabularx}
\caption{Initial segments of transforms at $r=1$ and $r=2$ for several classical sequences.}
\label{tab:initial-segments}
\end{table}

\section{Two concrete interpretations}

It is fair to say that \cref{thm:root-shift} already answers the most common ``iterated binomial transform'' question: once a sequence is known to satisfy a constant-coefficient recurrence, one can write down the transformed recurrence by the formal substitution \(X\mapsto X-r\).
However, the operators \(T_r\) are not merely a convenient notation for this substitution.
They encode two genuinely structural viewpoints that are often more useful than the recurrence itself.

The first one is combinatorial and becomes especially transparent when \(r\) is a nonnegative integer.
At the level of exponential generating functions, \cref{eq:egf} says that \(T_r\) simply multiplies by \(e^{rt}\), which is the exponential generating function of a labeled set of ``free'' atoms with \(r\) colors.

\begin{proposition}\label{prop:comb-interpretation}
Let \(r\in\mathbb{Z}_{\ge 0}\).
Assume that \(\seq{a}=(a_n)_{n\ge 0}\) is the counting sequence of a labeled class \(\mathcal{A}\), so that \(\widehat{A}(t)=\sum_{n\ge 0} a_n t^n/n!\) is its exponential generating function.
Then \(\seq{b}:=T_r\seq{a}\) counts the class of pairs consisting of an \(\mathcal{A}\)-structure on a subset of labels together with an \(r\)-coloring of the complementary labels.
Equivalently, for each \(n\), one has
\[
b_n=\sum_{k=0}^{n}\binom{n}{k} r^{\,n-k} a_k,
\]
where \(\binom{n}{k}\) chooses the subset carrying the \(\mathcal{A}\)-structure and \(r^{n-k}\) colors the remaining labels.
\end{proposition}

\begin{proof}
By \cref{eq:egf}, the exponential generating function of \(\seq{b}=T_r\seq{a}\) is
\(\widehat{A}_r(t)=e^{rt}\widehat{A}(t)\).
When \(r\in\mathbb{Z}_{\ge 0}\), the factor \(e^{rt}\) is the exponential generating function of a labeled set of atoms, each carrying one of \(r\) colors.
On an \(n\)-element label set, choose \(k\) labels to support an \(\mathcal{A}\)-structure (\(\binom{n}{k}\) choices), and color the remaining \(n-k\) labels in \(r^{n-k}\) ways.
Summing over \(k\) yields the displayed formula for \(b_n\), which is equivalent to the stated combinatorial description.
\end{proof}

This interpretation immediately explains why the semigroup law in \cref{prop:semigroup} is additive: ``adding \(r\) free colors'' and then ``adding \(s\) free colors'' is the same as adding \(r+s\) free colors at once.
It also suggests a robust way to generate new combinatorial models for familiar sequences by inserting a colored ``noise'' set.

\bigskip

The second viewpoint is linear-algebraic.
Many sequences of interest arise as matrix coefficients \(a_n=u^{\mathsf{T}}M^{n}v\), for instance from companion matrices of linear recurrences or from adjacency matrices of directed graphs.
In this setting, the transform \(T_r\) becomes a literal spectral shift \(M\mapsto M+rI\).

\begin{proposition}\label{prop:matrix-shift}
Let \(M\) be a square matrix over a commutative \(\mathbb{C}\)-algebra, and let \(u,v\) be compatible column vectors.
Define \(a_n:=u^{\mathsf{T}}M^{n}v\).
Then for every \(r\in\mathbb{C}\),
\[
(T_r\seq{a})_n = u^{\mathsf{T}}\left(M+rI\right)^{n}v.
\]
\end{proposition}

\begin{proof}
Since \(rI\) commutes with \(M\), the binomial theorem gives \(\left(M+rI\right)^{n}=\sum_{k=0}^{n}\binom{n}{k} r^{\,n-k} M^{k}\).
Multiplying by \(u^{\mathsf{T}}\) on the left and by \(v\) on the right yields the claim.
\end{proof}

In particular, when \(M\) is a companion matrix for a constant-coefficient recurrence, \cref{prop:matrix-shift} is another way to see the translation of characteristic roots in \cref{thm:root-shift}.
For graph-theoretic sequences, the same identity may be read as adding a loop of weight \(r\) at each vertex, turning each walk into a walk with ``lazy'' steps; for integer \(r\), these lazy steps can be interpreted as \(r\)-colored choices.

\section{Further remarks and directions}

The philosophy of this note is that many computations surrounding iterated binomial transforms are instances of a basic operator identity, namely \(T_r\circ T_s=T_{r+s}\).
Once this is in place, the cost of extracting concrete recurrences or closed forms becomes essentially the cost of a single substitution \(X\mapsto X-r\) in a characteristic polynomial. There are several natural extensions that can be pursued in the same ``structural first'' style.

At the level of exponential generating functions, \cref{eq:egf} states that \(T_r\) acts by multiplication by \(e^{rt}\).
Thus \(T_r\) is the simplest nontrivial element of the Sheffer group, and the semigroup law is the additive law of the exponent; see \cite{Roman} for background on the umbral/Sheffer calculus.
This viewpoint is efficient when studying the effect of \(T_r\) on polynomial sequences in \(n\) and its connections to Stirling and Bell numbers.

The same package of ideas applies to several closely related transforms, such as binomial transforms with alternating signs, and more generally to two-parameter binomial operators whose composition law explains iteration; see \cite{BarberoCerrutiMurru2010,AbrateBarberoCerrutiMurru2011}.
In each case, there is typically an underlying Riordan (or Sheffer) group element whose composition law explains iteration.

When binomial-type transforms are studied primarily as tools, it is usually best to isolate the operator identity and its composition law first, and to postpone sequence-specific computations to a compact corollary table.
In particular, \cref{tab:classical,tab:initial-segments} are intended as quick ``drop-in'' references: they allow one to recover explicit recurrences and to locate related objects in the OEIS without rereading the proofs.

\enlargethispage{\baselineskip}

\end{document}